\documentclass{article}
\usepackage{amssymb,amsmath}
\makeatletter
\newenvironment{proof}[1][\proofname]{\par
  \normalfont \topsep6\p@\@plus6\p@\relax
  \trivlist
  \item[\hskip\labelsep
        \itshape
    #1\@addpunct{.}]\ignorespaces
}{%
  \endtrivlist\@endpefalse
}
\providecommand{\proofname}{Proof}

\def\newtheoremnn#1#2{%
  \expandafter\@ifdefinable\csname #1\endcsname
    {\@definecounter{#1}%
     \expandafter\xdef\csname the#1\endcsname{}%
     \global\@namedef{#1}{\@thm{#1}{#2.}}%
     \global\@namedef{end#1}{\@endtheorem}}}

\makeatother

\newtheorem{e-proposition}[theorem]{Proposition}
\newtheorem{e-definition}[theorem]{Definition\rm}

\newtheoremnn{theoremnn}{Theorem}
\newtheoremnn{corollary}{Corollary}
\newtheoremnn{remark}{\it Remark\/}
\newtheoremnn{remarks}{\it Remarks\/}
\newtheoremnn{lemma}{Lemma}

\renewcommand{\tilde}{\widetilde}
\def\weakly{\rightharpoonup}
\def\weakto{\weakly}
\def\weaklystar{\stackrel{\ast}{\rightharpoonup}}

\def\Div{\operatorname{div}}

\def\curl{\operatorname{curl}}

\def\bbR{\mathbb{R}}
\def\R{\mathbb{R}}
\def\bbN{\mathbb{N}}

\newcommand{\hide}[1]{{}}
\newcommand{\dv}[1]{{\rm d}#1}

\newcommand{\MM}[2]{\mathbb{M}^{#1\times #2}}

\begin{document}
\begin{center}
  {  \LARGE
The div--curl lemma for sequences\\
whose divergence and curl are compact in $W^{-1,1}$}\\[5mm]
{\today}\\[5mm]
Sergio Conti$^1$,
Georg Dolzmann$^2$,
Stefan M\"uller$^{1,3}$\\[2mm]
{\em $^{1}$  
Institut f\"ur Angewandte Mathematik,
Universit\"at Bonn  \\
Endenicher Allee 60,
53115 Bonn, Germany\\[1mm]
$^2$ Universit\"at Regensburg, 93040 Regensburg, Germany\\[1mm]
$^3$Hausdorff Center for Mathematics, 
Universit\"at Bonn  \\
Endenicher Allee 60, 
53115 Bonn, Germany
}
\\[3mm]
  \begin{minipage}[c]{0.8\textwidth}
It is shown that $u_k \cdot v_k$ converges weakly to $u\cdot v$ if $u_k\weakto
u$ weakly in $L^p$ 
and $v_k\weakly v$ weakly in $L^q$ with $p,\,q\in(1,\infty)$, $1/p+1/q=1$,
under the additional assumptions that the sequences $\Div u_k$ 
and $\curl v_k$ are compact in the dual space of $W^{1,\infty}_0$ and
that $u_k\cdot v_k$ is equi-integrable.
The main point is that we only require equi-integrability of the scalar product
$u_k\cdot v_k$ and not of the individual sequences.
  \end{minipage}
\end{center}

\section{Statement of the Theorem}
The div--curl lemma is the cornerstone of the theory of compensated
compactness  
which was developed by Murat and Tartar in the late 
seventies~\cite{Murat1978,Murat1981,Tartar1978,Tartar1979,Tartar1982}, and is
still a very 
active area of research~\cite{BrianeCasadoDiazMurat2009}.
 In its classical
form the lemma states the following: if $\{u_k\}_{k\in\bbN}$ and $\{v_k\}_{k\in\bbN}$
are sequences in $L^2(\Omega;\bbR^n)$ which converge weakly in
$L^2(\Omega;\bbR^n)$ to $u$ and $v$, respectively, and if 
 $\Div u_k$ is compact in $H^{-1}(\Omega)$
and $\curl v_k$ is compact in $H^{-1}(\Omega;\MM{n}{n})$, then
\begin{alignat*}{1}
 u_k\cdot v_k \weakto
u\cdot v
\quad\text{ in }\mathcal D^\prime(\Omega)\,.
\end{alignat*}
A natural  generalization concerned sequences bounded in $L^p(\Omega;\bbR^n)$
and $L^q(\Omega;\bbR^n)$, respectively, where $p,\,q\in(1,\infty)$ are dual
exponents, $1/p+1/q=1$,  $\Div u_k$ is compact in $W^{-1,p}(\Omega)$
and $\curl v_k$ is compact in $W^{-1,q}(\Omega;\MM{n}{n})$, respectively, see \cite{Murat1981}.
Important connections to Hardy spaces were established in~\cite{CoifmanEtAl}.

This note is inspired by questions in nonlinear models in 
crystal plasticity~\cite{ContiDolzmannKlust2} in a two-dimensional setting.
The key point in this context is to prove that the determinant of the 
deformation gradient $\det\nabla \varphi_k$ converges to $\det\nabla \varphi$
under the assumption that 
$\nabla \varphi_k=G_k+B_k$ where $G_k\weakly \nabla \varphi$ weakly in $L^2$ and
$B_k\to 0$ strongly in $L^1$. The key additional information is that $\det \nabla
\varphi_k$ is compact in $L^1$.

Motivated by this application,
we present a generalization of the div--curl lemma 
with very weak assumptions on
$\Div u_k$ and $\curl v_k$ and the additional assumption that $u_k\cdot v_k$
is equi-integrable (see the remarks after the Theorem). We denote the dual of 
 $W_0^{1,\infty}(\Omega)$ by $W^{-1,1}(\Omega)$. 

\begin{theoremnn}
  Let $\Omega\subset\R^n$ be an open and  bounded domain with Lipschitz boundary and
let  $p,\,q\in(1,\infty)$ with $1/p+1/q=1$.
Suppose $u_k\in L^p(\Omega;\R^n)$,
  $v_k\in L^q(\Omega;\R^n)$ are sequences
  such that
  \begin{equation}
 u_k\weakto u \hskip3mm\text{ weakly in  $L^p(\Omega;\R^n)$}
\hskip3mm \text{ and }\hskip3mm
 v_k\weakto v \hskip3mm\text{ weakly in  $L^q(\Omega;\R^n)$}\,,
  \end{equation}
and
\begin{equation}\label{equkvkequiint}
  u_k\cdot v_k \text{ is equi-integrable}.
\end{equation}
Finally assume that
\begin{equation}
  \Div u_k\to \Div u \hskip3mm\text{ in $W^{-1,1}(\Omega)$} \hskip3mm \text{ and }\hskip3mm
\curl v_k\to \curl v \hskip3mm\text{ in $W^{-1,1}(\Omega;\MM{n}{n})$}\,.
\end{equation}
Then
\begin{equation}
  u_k\cdot v_k \weakly u\cdot v \hskip1cm \text{ weakly in } L^1(\Omega)\,.
\end{equation}
\end{theoremnn}
\begin{remarks}
  \begin{enumerate}
  \item The statement is almost classical under the stronger hypothesis that
    $|u_k|^p$ and $|v_k|^q$ are equi-integrable (see the Lemma below). The main
    novelty is that here we require only that $u_k\cdot v_k$ is
    equi-integrable, and this is crucial for the application in
    \cite{ContiDolzmannKlust2}.
    \item The assumption that the inner product $u_k\cdot v_k$  is equi-integrable is
necessary as can be seen from the one dimensional example of a Fakir's
carpet. Let  $u_k=v_k$ be given on the unit interval by
$u_k =  \sqrt{k}\sum_{\ell=1}^k \chi_{[\ell/k,k^{-2}+\ell/k ]}$.
Then $u_k$ converges to zero weakly  in $L^2(0,1)$ and strongly in $L^1(0,1)$, 
but $u_k^2$ converges to one in the sense of distributions. 
  \end{enumerate}
\end{remarks}
The crucial observation in the proof is the fact that given
(\ref{equkvkequiint}) we can construct modified sequences  $\tilde u_k$ and
$\tilde v_k$ such that $\tilde u_k\cdot \tilde v_k$ has the same weak limit as
$u_k\cdot v_k$ 
and the sequences $|u_k|^p$ and $|v_k|^q$ are equi-integrable and therefore compact
in $W^{-1,p}$ and $W^{-1,q}$, respectively. The sequences are constructed
using  the biting 
lemma~\cite{BrooksChacon1980,BallMurat1989} and Lipschitz
truncations of Sobolev functions which originate in the work of Liu~\cite{Liu1977}
and Acerbi and Fusco~\cite{AcerbiFusco1984,AcerbiFusco1988} and have found
important applications in the vector-valued calculus of variations, see, e.g.,
\cite{BallZhang1990,Zhang1992,SM99Zhang}. 
In two dimensions, a change of variables leads to weak continuity of the determinant:
\begin{corollary}
 Let $\Omega\subset\R^2$ be an open and bounded domain with Lipschitz
 boundary, and let $\varphi_k\in W^{1,1}(\Omega;\R^2)$ be such that 
$\det\nabla \varphi_k=G^k+B^k$, with $B^k\to0$ strongly in $L^1$ and
$G^k\weakto G$ weakly in $L^2$. If the sequence $\det
 \nabla\varphi_k$ is equi-integrable, then $\det\nabla
 \varphi_k\weakto\det G$ 
 weakly in $L^1$.
\end{corollary}

\section{Proofs}

We begin with the proof of the lemma that shows how equi-integrability of
$|u_k|^p$  leads to compactness of $\Div u_k$. We say that a sequence 
 $u_k\in L^p(\Omega;\R^n)$ is $L^p$-equi-integrable if
there is an increasing function
$\omega:[0,\infty) \to\R$ with $\lim_{t\to 0} \omega(t)=0$, such that
\begin{equation}\label{eqdeffequiint}
  \int_A |u_k|^p \, \dv{x}\le \omega(t) \hskip1cm \text{ for all
    $A\subset\Omega$     measurable 
    with $|A|\le t$}\,.
\end{equation}
\begin{lemma}\label{LemmaH1}
  Let $\Omega\subset\R^n$ be a bounded Lipschitz set, $1<p<\infty$, 
and let $u_k\in L^p(\Omega;\R^n)$
  be an $L^p$-equi-integrable sequence. 
If  $\Div u_k\to 0$ in $W^{-1,1}(\Omega)$, then $\Div u_k\to 0$ in
  $W^{-1,p}(\Omega)$. The analogous statements hold for $\curl u_k$ and $\nabla u_k$.
\end{lemma}
\begin{proof}
Let $\omega$ be as in (\ref{eqdeffequiint}).
  By definition and density of $C_0^\infty(\Omega)$ in $W^{1,q}_0(\Omega)$, 
  \begin{equation}
    \|\Div u_k\|_{W^{-1,p}(\Omega)}=\sup \left\{ \int_\Omega \nabla
      \varphi\cdot u_k \, \dv{x} 
\colon \varphi\in
      C_0^\infty(\Omega), \int_\Omega|\nabla \varphi|^q \dv{x}\le 1\right\}\,,
  \end{equation} 
  where $q$ is given by $1/p+1/q=1$.
  Fix $\varphi\in C_0^\infty(\Omega)$ with $\|\nabla\varphi\|_q\le 1$ and $t>0$. By the
  truncation argument in~\cite[Lemma 4.1]{DolzmannHungerbuehlerMueller2000} 
or~\cite[Prop. A.2]{FJM02}
   there is a $t$-Lipschitz function $\psi\in
   W^{1,\infty}_0(\Omega)$ such that the measure of the set
   $M=\{ \psi\ne \varphi \text{ or } \nabla\psi\ne \nabla\varphi\}$ is bounded by $c_*/t^q$, where $c_*$  depends 
  only on  $\Omega$. We decompose
  \begin{equation}
    \int_\Omega \nabla \varphi\cdot u_k \, \dv{x}= 
    \int_\Omega (\nabla \varphi-\nabla\psi)\cdot u_k \, \dv{x}+\int_\Omega
    \nabla\psi\cdot u_k\, \dv{x}\,.
  \end{equation}
  The second term is bounded by
  $\|\nabla\psi\|_{L^\infty}\|\Div u_k\|_{W^{-1,1}}$.
  The first term is   concentrated on the set $M$, and by H\"older's 
inequality can be estimated by
  \begin{alignat}{1}
\int_M (\nabla \varphi-\nabla\psi)\cdot u_k\, \dv{x}
    &\le \left(\int_M (|\nabla\varphi|+t)^q\, \dv{x}\right)^{1/q} 
\left(\int_M |u_k|^p\, \dv{x}\right)^{1/p}\,.
  \end{alignat}
  The first factor is bounded by 
$\|\nabla \varphi\|_{L^q(M)}+|M|^{1/q}t\leq 1+c_\ast^{1/q}$, the second by
$\left(\omega(c_* t^{-q}) \right)^{1/p} $ in view of 
the equi-integrability of the sequence $|u_k|^p$, and we conclude that
  \begin{equation}
    \|\Div u_k\|_{W^{-1,p}(\Omega)}\le \left( 1+c_\ast^{1/q} \right)\,
\left(\omega(c_* t^{-q}) \right)^{1/p} 
    + t \|\Div u_k\|_{W^{-1,1}(\Omega)}\,,
  \end{equation}
    with $\omega$ as in (\ref{eqdeffequiint}). 
The assertion follows with $t=\|\Div u_k\|_{W^{-1,1}(\Omega)}^{-1/2}$.
\end{proof}

\begin{proof}[Proof of the Theorem]
We divide the proof into four steps. The first three treat the
case $u=v=0$.\\
{\em Step 1: Modification of  $u_k$ and $v_k$ to obtain $L^p$ and $L^q$--equi-integrable sequences, respectively.}
The sequence $|u_k|^p$ is bounded in $L^1$, and therefore the biting lemma
\cite{BallMurat1989,SaadouneValadier1995}
implies the existence of a sequence of sets $A_k\subset \Omega$ such 
that $|A_k|\to0$ and, after extracting a subsequence, $|u_k|^p\chi_{\Omega\setminus A_k}$ is
equi-integrable. Set $\tilde u_k=u_k \chi_{\Omega\setminus
  A_k}$.
Since  $\|\tilde u_k-u_k\|_{L^1(\Omega)}=\|u_k\|_{L^1(A_k)}\le |A_k|^{1/q}\|u_k\|_{L^p(\Omega)}$
it follows that
\begin{equation}
\tilde u_k-u_k\to 0\hskip5mm\text{ in }L^1(\Omega)\,.
\end{equation}
Therefore the two sequences $u_k$, $\tilde u_k$ have the
same weak limit (in $L^p$). Furthermore, $\nabla (\tilde u_k-u_k)\to0$ in
$W^{-1,1}(\Omega;\MM{n}{n})$, and therefore $\Div \tilde u_k\to0$ in
$W^{-1,1}(\Omega)$.   One proceeds analogously with $v_k$, obtains the
corresponding sets $B_k$ and
a sequence $\tilde v_k=v_k\chi_{\Omega\setminus B_k}$.
To conclude this step it  remains to prove that
$u_k\cdot v_k-\tilde u_k\cdot \tilde v_k\weakly 0$ in $L^1$. To see
this, we observe that this expression vanishes outside of $A_k\cup B_k$, and that
it equals $u_k\cdot v_k$ on this set. By equi-integrability of $u_k\cdot v_k$
 and the fact that $|A_k\cup B_k|\to0$, we conclude that 
$u_k\cdot v_k-\tilde u_k\cdot \tilde v_k\to0$ in $L^1$. \\
{\em Step 2: Strong $W^{-1,p}$ convergence and reduction to the classical
  div-curl Lemma.} 
The sequence $\tilde u_k$ is $L^p$-equi-integrable, and its divergence
converges strongly to zero in $W^{-1,1}$. Therefore by the Lemma 
 we
obtain that $\Div \tilde u_k\to 0$ in $W^{-1,p}(\Omega)$. Analogously one shows
that $\curl \tilde v_k\to0$ in $W^{-1,q}(\Omega)$. By the classical div-curl
Lemma we then conclude that $\tilde u_k\cdot \tilde v_k \weaklystar 0$
in $\mathcal D^\prime(\Omega)$.\\
%
%
%
{\em Step 3: Identification of the $L^1$-weak limit.}
Since the sequence $u_k\cdot v_k$ is by assumption
equi-integrable it has a subsequence which converges
weakly in $L^1$. The same holds for $\widetilde u_k\cdot \widetilde v_k$. But the two limits are
the same (Step 1) and the latter is zero (Step 2). Since the limit does not
depend on the subsequence, the entire sequence converges. This concludes the
proof if $u=v=0$.\\ 
{\em Step 4: General case.}
We set $\tilde u_k=u_k-u$, $\tilde v_k=v_k-v$. Equi-integrability of the sequence $\tilde u_k\cdot
\tilde v_k$ follows from $\int_A |u_k\cdot v| \, \dv{x} \le
\|u_k\|_{L^p(\Omega)} \|v\|_{L^q(A)}$ for all $A\subset\Omega$ (and analogously
for $u\cdot v_k$). By Steps 1-3,
  $\tilde u_k\cdot \tilde v_k\weakto 0$ weakly in $L^1(\Omega)$.
The proof is concluded observing that 
$u_k\cdot v$ and $u\cdot v_k$ converge weakly   in $L^1$ to $u\cdot v$.
\end{proof}
\begin{proof}[Proof of the Corollary]
  Let $u_k=(e_1\cdot G^k)^\perp=(-G^k_{12}, G^k_{11})$, $v_k=e_2\cdot 
  G^k=(G^k_{21}, G^k_{22})$, so that   $\det
  G^k=u_k\cdot  v_k$. Since 
$G^k+B^k$ is a gradient, $\Div u_k=\partial_1 B^k_{12}-\partial_2 B^k_{11}$,
and therefore $\|\Div u_k\|_{W^{-1,1}}\le \|B_k\|_{L^1}\to0$. The same estimate
holds for  $\curl v_k$. At this point the Corollary follows from the Theorem.
\end{proof}

\section*{Acknowledgements}
This work was partially supported by the Deutsche Forschungsgemeinschaft
through FOR 797  {\em Analysis and computation of
  microstructure in finite plasticity},  
projects Co304/4-1, Do633/2-1, Mu1067/9-1. 

\providecommand{\bysame}{\leavevmode\hbox to3em{\hrulefill}\thinspace}
\providecommand{\MR}{\relax\ifhmode\unskip\space\fi MR }
\providecommand{\MRhref}[2]{%
  \href{http://www.ams.org/mathscinet-getitem?mr=#1}{#2}
}
\providecommand{\href}[2]{#2}

\end{document}